\documentclass[12pt]{amsart}
\usepackage{amssymb,bbold}
\usepackage[all]{xy}

\theoremstyle{plain}
\newtheorem{theorem}{Theorem}
\newtheorem{corollary}[theorem]{Corollary}
\newtheorem{lemma}[theorem]{Lemma}
\newtheorem{proposition}[theorem]{Proposition}

\theoremstyle{definition}

\newtheorem{remark}[theorem]{Remark}

\begin{document}
\baselineskip 18pt

\title[A topology on the Fremlin tensor product]
      {A topology on the Fremlin tensor product of locally convex-solid vector lattices}

\author[O.~Zabeti]{Omid Zabeti}
\address[O.~Zabeti]
  {Department of Mathematics, Faculty of Mathematics, Statistics, and Computer science,
   University of Sistan and Baluchestan, Zahedan,
   P.O. Box 98135-674. Iran}
\email{o.zabeti@gmail.com}
\keywords{ Fremlin tensor product, Fremlin projective tensor product, locally convex-solid vector lattice.}
\subjclass[2020]{Primary:  46M05. Secondary:  46A40.}
\maketitle

\begin{abstract}
Suppose $E$ and $F$ are locally convex-solid vector lattices. Although we have a suitable vector lattice structure for the tensor product $E$ and $F$ (known as the Fremlin tensor product and denoted by $E\overline{\otimes}F$), there is a lack of topological structure on $E\overline{\otimes}F$, in general.  In this note, we consider a topological attitude on $E\overline{\otimes}F$ that makes it into a locally convex-solid vector lattice, as well. 
\end{abstract}
\date{\today}

\maketitle
\section{Motivation and introduction}
Assume that $E$ and $F$ are Archimedean vector lattices. Fremlin in \cite{Fremlin:72}, constructed an Archimedean vector lattice $E\overline{\otimes}F$ that contains the algebraic tensor product $E\otimes F$ (considered as an ordered vector subspace). Now, suppose that $E$ and $F$ are Banach lattices. Fremlin in \cite{Fremlin:74}, construced a Banach lattice $E\widehat{\otimes}F$ that contains both $E\otimes F$ and $E\overline{\otimes}F$ as norm dense subspaces). Howover, when $E$ and $F$ are locally convex-solid vector lattices, there is a lack of topological structure on the Fremlin tensor product $E\overline{\otimes}F$ that makes it into a locally convex-solid vector lattice. In this note, we establish a locally convex-solid topology on $E\overline{\otimes}F$ using the topological structures in both $E$ and $F$. Moreover, we consider a seminorm apporoach to this setting. Before we state the main results, we recall some notes regarding Fremlin tensor products between vector and Banach lattices; for more details, see \cite{Fremlin:72, Fremlin:74}. Moreover, for a comprehensive, new and interesting reference, see \cite{Wickstead1:24}. Furthermore, for a short and nicely written exposition on different types of tensor products between Archimedean vector lattices, see \cite{Gr:23}.

Suppose $E$ and $F$ are Archimedean vector lattices. In 1972, Fremlin constructed a tensor product $E\overline{\otimes} F$ that is an  Archimedean vector lattice with the following properties:
\begin{itemize}
\item {The algebraic tensor product $E\otimes F$ is a vector subspace of $E\overline{\otimes}F$ so that it is an ordered vector subspace in its own right}.
\item{The vector sublattice in $E\overline{\otimes} F$ generated by $E\otimes F$ is the whole of $E\overline{\otimes}F$.}
\item{For each Archimedean vector lattice $G$ and every lattice bimorphism $\Phi:E\times F\to G$, there is a unique lattice homomorphism $T:E\overline{\otimes}F\to G$ such that $T(x\otimes y)=\Phi(x,y)$ for each $x\in E$ and for each $y\in F$. }
\end{itemize}
Therefore, every element of $E\overline{\otimes}F$ can be considered as a finite supremum and finite infimum of some elements of $E\otimes F$. The good news is that we have some density properties for elements of $E\overline{\otimes}F$ in terms of the elements of the algebraic tensor product $E\otimes F$ as follows.
\begin{itemize}
\item{Assume that $E$ and $F$ are Archimedean vector lattice. For each $a\in E\overline{\otimes}F$, there exists $u\in E_{+}$ and $v\in F_{+}$ such that for each $\varepsilon>0$, there is $b\in E\otimes F$ with $|a-b|\leq \varepsilon u\otimes v$, \cite[Proposition 3.11]{Wickstead1:24}.}
\item{For each $c\in (E\overline{\otimes}F)_{+}$, we have $c=\sup\{a\otimes b: a\in E_{+}, b\in F_{+}\}$, \cite[Proposition 3.12]{Wickstead1:24}.}
\item{For each $u\in (E\overline{\otimes}F)_{+}$, there exist $a\in E_{+}$ and $b\in F_{+}$ with $u\leq a\otimes b$, \cite[1A(d)]{Fremlin:74}.}
\end{itemize}
 Now, suppose $E$ and $F$ are Banach lattices. Fremlin in \cite{Fremlin:74} constructed a tensor product $E\widehat{\otimes}F$ that is a Banach lattice with the following properties.
 \begin{itemize}
 \item{$E\widehat{\otimes}F$ is the norm completion of $E\otimes F$ with respect to the projective norm: for each $u=\Sigma_{i=1}^{n}x_i\otimes y_i \in E\otimes F$:
 \[\|u\|_{|\pi|}=\sup\{|\Sigma_{i=1}^{n}\phi(x_i,y_i)|: \phi\hspace{0.25cm} \textit{is a bilinear form on}\hspace{0.35cm} E\times F \textit{and} \hspace{0.25cm}\|\phi\|\leq 1.\}\]}
 \item{$E\overline{\otimes}F$ can be considered as an norm-dense vector sublattice of $E\widehat{\otimes}F$.}
 \item{For every Banach lattice $G$ and every continuous bilinear mapping $\Phi:E\times F\to G$, there is a unique positive linear mapping $T:E\widehat{\otimes}F\to G$ with $T(x\otimes y)=\Phi(x,y)$. Furthermore, $\Phi$ is a lattice bimorphism if and only if $T$ is a lattice homomorphism.}
 \item{The positive cone in $E\widehat{\otimes}F$ is the closure of the cone $P\subseteq E\otimes F$ generated by $\{x\otimes y: x\in E_{+}, y\in F_{+}\}$.}
 \item{The projective norm, $\|.\|_{|\pi|}$, on $E\widehat{\otimes}F$ is a cross norm that is for every $x\in E$ and for every $y\in F$, we have $\|x\otimes y\|_{|\pi|}=\|x\|\|y\|$.}
 \end{itemize}
 For undefined terminology and general theory of vector lattices, Banach lattices and also locally convex-solid vector lattices, see \cite{AB1, AB}. Also, for new and recent achievments regarding tensor products in vector lattices, see \cite{Gr:23, Wickstead1:24}.
 \section{main results}
 Suppose $E$ is an Archimedean vector lattice and $A\subseteq E$. We denote by $Sol(A)$, the smallest solid set of $E$ containing $A$. More precisely
  \[Sol(A)=\{x\in E, \exists y\in A, |x|\leq |y|.\}\]
  We denote by $Conv(A)$, the convex hull of $A$ which is the smallest convex set of $E$ containing $A$. In other words,
  \[Conv(A)=\{\Sigma_{i=1}^{n}\lambda_ix_i: n\in \Bbb N, \lambda_i\geq 0, x_i\in A, \Sigma_{i=1}^{n}\lambda_i=1.\}\]
  Also, the convex balanced hull of $A$ is defined as follows.
  \[Conv_b(A)=\{\Sigma_{i=1}^{n}\lambda_ix_i: n\in \Bbb N, \lambda_i\in \Bbb R, x_i\in A, \Sigma_{i=1}^{n}|\lambda_i|\leq1.\}\]
  
  It is routine to verify that if $A\subseteq B$, then, $Sol(A)\subseteq Sol(B)$, $Conv(A)\subseteq Conv(B)$ and $Conv_b(A)\subseteq Conv_b(B)$.
  
  Moreover, for two sets $A,B\subseteq E$, we denote by $A\vee B$ and $A\wedge B$ the sets $\{a\vee b:a\in A \hspace{0.25cm}\text{and}\hspace{0.25cm} b\in B\}$ and $\{a\wedge b:a\in A \hspace{0.25cm}\text{and}\hspace{0.25cm} b\in B\}$, respectively.
  The following results are straightforward; nevertheless, necessary in the whole of the note. We give a proof for the sake of completeness.
 \begin{lemma}\label{1}
 Suppose $E$ is an Archimedean vector lattice and $A,B\subseteq E$. Then, we have the following observations.
 \begin{itemize}
  \item[\em(1)]{$Conv(A+B)=Conv(A)+Conv(B)$.}
   \item[\em(2)]{$Conv_b(A+B)=Conv_b(A)+Conv_b(B)$.}
   \item[\em (3)]{$Conv_b(A\cup B)\subseteq Conv_b(A)\cup Conv_b(B)$; the same holds for $Conv$.}
   \item[\em(4)]{$Conv_b(A\cap B)\subseteq Conv_b(A)\cap Conv_b(B)$; the same holds for $Conv$.}
 \item[\em(5)]{$Sol(A+B)\subseteq Sol(A)+Sol(B)$.}
 \item[\em(6)]{$Sol(\alpha A)=\alpha Sol(A)$ and $Conv_b(\alpha A)=\alpha Conv_b(A)$, for each real $\alpha$; the same holds for $Conv$.}
 \item[\em(7)]{$Sol(A\cup B)\subseteq Sol(A)\cup Sol(B)$.}
 \item[\em(8)]{$Sol(A\cap B)\subseteq Sol(A)\cap Sol(B)$.}
 \item[\em(9)]{$Sol(A\vee B)\subseteq Sol(A)\vee Sol(B)$.}
 \item[\em(10)]{$Sol(A\wedge B)\subseteq Sol(A)\wedge Sol(B)$.}
 \item[\em(11)]{If $F$ is also a vector lattice and $T:E\to F$ is a lattice homomorphism, then, $T(Sol(A))\subseteq Sol(T(A))$.}
 \end{itemize}
 \end{lemma}
 \begin{proof}
The proofs of $(1)$ and $(2)$ have essentially the same idea; we prove just $(ii)$. Since $Conv_b(A)+Conv_b(B)$ is a convex and balanced set that contains $A+B$, we conclude that $Conv_b(A+B)\subseteq Conv_b(A)+Conv_b(B)$. For the other side, assume that
 $u=\Sigma_{i=1}^{n}\alpha_i a_i \in Conv_b(A)$  and   $v=\Sigma_{j=1}^{m}\beta_j b_j \in Conv_b(B)$ with  $\Sigma_{i=1}^{n}|\alpha_i|\leq 1$ and $\Sigma_{j=1}^{m}|\beta_j|\leq 1$. We can write $u+v=\Sigma_{i=1}^{n}\Sigma_{j=1}^{m}\alpha_i\beta_j(a_i+b_j)$ so that $u+v\in Conv_b(A+B)$.


It can be seen easily that intersection and union of convex and balanced sets are again convex and balanced so that we conclude that $(3)$ and $(4)$ hold. 

 To prove $(5)$, assume that $z\in Sol(A+B)$. There exists $x\in A$ and $y\in B$ with $|z|\leq |x+y|$. By the Reisz decomposition property(\cite[Theorem 1.10]{AB1}), there exist $z_1\in E$ and $z_2\in E$ with $z=z_1+z_2$, $|z_1|\leq |x|$ and $|z_2|\leq |y|$. Thus, $z_1\in Sol(A)$ and $z_2\in Sol(B)$ so that $z\in Sol(A)+Sol(B)$.

 $(6)$. Note that
 \[Sol(\alpha A)=\{z\in E, \exists x\in A, |z|\leq |\alpha||x|\}=\alpha\{\frac{z}{\alpha}\in E, \exists x\in A, \frac{|z|}{|\alpha|}\leq |x|\}=\alpha Sol(A).\]
The statement for the convex hull can be obtained from the definition, easily.

 $(7)$. Assume that $z\in Sol(A\cup B)$. There exists $x\in A\cup B$ with $|z|\leq |x|$ so that $x\in A$ or $x\in B$. This means that $z\in Sol(A)$ or $z\in Sol(B)$.
 
 $(8)$. Suppose $z\in Sol(A\cap B)$. There exists $x\in A\cap B$ with $|z|\leq |x|$ so that $x\in A$ and $x\in B$. This means that $z\in Sol(A)\cap Sol(B)$.
 
 $(9)$. Assume that $z\in Sol(A\vee B)$. There exist $x\in A$ and $y\in B$ with $|z|\leq |x\vee y|\leq |x|\vee |y|$. W.O.L.G, we can assume that $|x|\wedge |y|=0$; otherwise, write $x'=|x|-|x|\wedge |y|$ and $y'=|y|-|x|\wedge |y|$. Then, $|x|\vee |y|=x'\vee y'$ and $x'\wedge y'=0$. Moreover, $x'\in Sol(A)$ and $y'\in Sol(B)$. So, $|z|\leq |x|+|y|$.  By the Riesz decomposition property, There are positive elements $z_1,z_2$ in $E$ with $|z_1|\leq |x|$, $|z_2|\leq |y|$ and $z=z_1+z_2$. Note that $z_1$ and $z_2$ are also disjoint as $z_1\wedge z_2\leq |x|\wedge |y|=0$. Therefore, $z=z_1\vee z_2$, $z_1\in Sol(A)$ and $z_2\in Sol(B)$.
 
 $(10)$. By using $(6)$ and $(9)$, we have
 \[Sol(A\wedge B)=Sol(-((-A)\vee (-B)))=-Sol((-A)\vee(-B))\subseteq -(Sol(-A)\vee Sol(-B))=\]
 \[-(-Sol(A)\vee -Sol(B))=Sol(A)\wedge Sol(B).\]
$(11)$. Suppose $F$ is a vector lattice, $T:E\to F$ is a lattice homomorphism and $w\in T(Sol(A))$. There exists $u\in Sol(A)$ with $w=T(u)$. Moreover, we can find $z\in A$ such that $|z|\leq |u|$. We have
\[|T(z)|=T(|z|)\leq T(|u|)=|T(u)|=|w|.\]
This means that $w\in Sol(T(A))$. 
   \end{proof}
   \begin{corollary}
   Suppose $E$ is an Archimedean vector lattice and $A,B\subseteq E$. If $A$ and $B$ are solid (convex), then, so are $A+B$, $A\vee B$, $A\wedge B$, $A\cup B$ and $A\cap B$.
   \end{corollary}
   Now, we introduce a locally convex-solid topology on the Fremlin tensor product of two locally convex-solid vector lattices $E$ and $F$.  
   \begin{theorem}\label{2}
   Suppose $E$ and $F$ are locally convex-solid vector lattices. Then, there exists a locally convex-solid topology on the Fremlin tensor product $E\overline{\otimes}F$. 
   \end{theorem}
   \begin{proof}
   Assume that $\{U_{\alpha}: \alpha\in I\}$ and $\{V_{\beta}:\beta \in J\}$ are bases of convex-solid zero neighborhoods for the corresponding topologies in $E$ and $F$, respectively. Put
   \[\textbf{B}=\{Conv_b(Sol(U_{\alpha}\otimes V_{\beta})): \alpha\in I, \beta\in J.\}\]
   We claim that $\textbf{B}$ is a basis of convex-solid zero neighborhoods for a topology on $E\overline{\otimes}F$. We use the procedure of \cite[Section 7]{Den:17}. For each $\alpha \in I$ and for each $\beta \in J$, put $W_{\alpha,\beta}=Conv_b(Sol(U_{\alpha}\otimes V_{\beta}))$ which is again solid by \cite[Theorem 1.11]{AB1}, in which $U_{\alpha}\otimes V_{\beta}=\{x\otimes y: x\in U_{\alpha}, y\in V_{\beta}\}$. Note that \cite[Theorem 1.11]{AB1} proves that the convex hull of a solid set is also solid but it is an easy matter to see that it is also true for the convex balanced hull, as well.  We show that intersection of every two elements in $\textbf{B}$ contains another element of $\textbf{B}$. For each $\alpha,\alpha' \in I$, choose any zero neighborhood $U\subseteq E$ with $U\subseteq U_{\alpha}\cap U_{\alpha'}$. Similarly, find a zero neighborhood $V\subseteq F$ with $V\subseteq V_{\beta}\cap V_{\beta'}$. Take  zero neighborhoods $U_{\alpha_0}$ and $V_{\beta_0}$ with $U_{\alpha_0}\subseteq U$ and $V_{\beta_0}\subseteq V$.  Put $W=Conv_b(Sol(U_{\alpha_0}\otimes V_{\beta_0}))$. We claim that $W_{\alpha_0,\beta_0}\subseteq W_{\alpha,\beta}\cap W_{\alpha',\beta'}$. By using Lemma \ref{1}, we have
   \[Conv_b(Sol(U_{\alpha_0}\otimes V_{\beta_0}))\subseteq Conv_b(Sol((U_{\alpha}\cap U_{\alpha'})\otimes (V_{\beta}\cap V_{\beta'})))=Conv_b(Sol((U_{\alpha}\otimes V_{\beta})\cap(U_{\alpha'}\otimes V_{\beta'})))\subseteq\]
   \[Conv_b(Sol(U_{\alpha}\otimes V_{\beta}))\cap Conv_b(Sol(U_{\alpha'}\otimes V_{\beta'}))=W_{\alpha,\beta}\cap W_{\alpha',\beta'}.\]
   
   Take any $W_{\alpha,\beta}$. We need to find $W_{\alpha_0,\beta_0}$ such that $W_{\alpha_0,\beta_0}+W_{\alpha_0,\beta_0}\subseteq W_{\alpha,\beta}$. There are convex solid zero neighborhoods $U_{\alpha_0}\subseteq E$ and $V_{\alpha_0}\subseteq F$ with $U_{\alpha_0}\subseteq \frac{1}{\sqrt{2}}U_{\alpha}$ and $V_{\beta_0}\subseteq \frac{1}{\sqrt{2}}V_{\beta}$. 
   Therefore,
   again, by using Lemma \ref{1}, we have
   \[W_{\alpha_0,\beta_0}+W_{\alpha_0,\beta_0}=Conv_b(Sol(U_{\alpha_0}\otimes V_{\beta_0}))+Conv_b(Sol(U_{\alpha_0}\otimes V_{\beta_0}))\subseteq\]\[ Conv_b(Sol(\frac{1}{\sqrt{2}}U_{\alpha}\otimes \frac{1}{\sqrt{2}}V_{\beta}))+Conv_b(Sol(\frac{1}{\sqrt{2}}U_{\alpha}\otimes \frac{1}{\sqrt{2}}V_{\beta}))=\]
   \[\frac{1}{2}Conv_b(Sol(U_{\alpha}\otimes V_{\beta}))+\frac{1}{2}Conv_b(Sol(U_{\alpha}\otimes V_{\beta}))=Conv_b(Sol(U_{\alpha}\otimes V_{\beta}))= W_{\alpha,\beta}.\]
   Moreover, for each real $\lambda$ with $|\lambda|\leq 1$, (by considering Lemma \ref{1}), we have 
   \[\lambda W_{\alpha,\beta}=\lambda Conv_b(Sol(U_{\alpha}\otimes V_{\beta}))=Conv_b(Sol(\lambda U_{\alpha}\otimes V_{\beta}))\subseteq Conv_b(Sol(U_{\alpha}\otimes V_{\beta}))=W_{\alpha,\beta}.\]
   
  Finally, we need to prove that for each $W_{\alpha,\beta}$ and for each $z\in W_{\alpha,\beta}$, there exists some $W_{\alpha_0,\beta_0}$ with $z+W_{\alpha_0,\beta_0}\subseteq W_{\alpha,\beta}$. We can write $z=\Sigma_{i=1}^{n}\lambda_i z_i$, in which $\Sigma_{i=1}^{n}|\lambda_i|\leq1$ and $z_i\in Sol(U_{\alpha}\otimes V_{\beta})$. So, we can find $x_i\in U_{\alpha}$ and $y_i\in V_{\beta}$ such that $|z_i|\leq |x_i\otimes y_i|=|x_i|\otimes |y_i|$. Also, there are convex-solid zero neighborhoods $U_{\alpha_i}$ in $E$ and $V_{\beta_i}$ in $F$ with $|x_i|+U_{\alpha_i}\subseteq U_{\alpha}$ and $|y_i|+V_{\beta_i}\subseteq V_{\beta}$. Note that we can choose also convex-solid zero neighborhoods $U_{\alpha_0}$ and $V_{\beta_0}$ with $U_{\alpha_0}\subseteq \cap_{i=1}^{n}U_{\alpha_i}$ and $V_{\beta_0}\subseteq \cap_{i=1}^{n}V_{\beta_i}$. We claim that $z+W_{\alpha_0,\beta_0}\subseteq W_{\alpha,\beta}$. Suppose $w=\Sigma_{j=1}^{m}\gamma_j w_j$, in which $\Sigma_{j=1}^{m}|\gamma_j|\leq1$ and $w_j\in Sol(U_{\alpha_0}\otimes V_{\beta_0})$. So, we can find $u_j\in U_{\alpha_0}$ and $v_j\in V_{\beta_0}$ with $|w_j|\leq |u_j|\otimes |v_j|$. It is easy to see that we can write $z+w=\Sigma_{i,j}\lambda_i\gamma_j (z_i+w_j)$ and note that $\Sigma_{i,j}|\lambda_i \gamma_j|\leq1$. Moreover,
  \[|z_i+ w_j|\leq|z_i|+ |w_j|\leq (|x_i|\otimes|y_i|)+(|u_j|\otimes |v_j|)\leq (|x_i|+|u_j|)\otimes (|y_i|+ |v_j|)\in U_{\alpha}\otimes V_{\beta}.\]
  This means that $z+w\in W_{\alpha,\beta}$. 
  
  So, we have a linear topology with a base consisting of all convex-solid zero neighborhoods; denoted by $\tau_F$. So, $(E\overline{\otimes}F, \tau_F)$ is a locally convex-solid vector lattice.
  
   \end{proof}
   \begin{remark}
   Suppose $(E,\tau_1)$ and $(F,\tau_2)$ are locally convex-solid vector lattices. By \cite[Theorem 2.25]{AB1}, every locally convex-solid topology is generated by a family of Riesz seminorms. Assume that $\tau_1$ is generated by a family $(p_{\alpha})_{\alpha\in I}$ and $\tau_2$ is generated by a family $(q_{\beta})_{\beta \in J}$ of Riesz seminorms. Define the real-valued function $p_{\alpha}\otimes q_{\beta}$ on $E\overline{\otimes} F$ via $(p_{\alpha}\otimes q_{\beta})(u)=\inf\{r>0: u\in r W_{\alpha,\beta}\}$. It can be easily verified that $p_{\alpha}\otimes q_{\beta}$ is a Riesz seminorm on $E\overline{\otimes}F$. Moreover, the topology generated by the family $(p_{\alpha}\otimes q_{\beta})_{\alpha\in I, \beta \in J}$ coincides with $\tau_F$. 
   \end{remark}
   Now we have the following equivalent definition of the topology $\tau_F$. 

   \begin{theorem}\label{3}
   Suppose $E$ and $F$ are locally convex-solid vector lattices with the generating families of Riesz seminorms $(p_{\alpha})_{\alpha\in I}$ and  $(q_{\beta})_{\beta \in J}$, respectively. Then, we have 
   \[(p_{\alpha}\otimes q_{\beta})(u)=\inf\{\Sigma_{i=1}^{n}p(x_i)q(y_i), x_i,y_i\geq 0, |u|\leq \Sigma_{i=1}^{n}x_i\otimes y_i\}.\]
   \end{theorem}
   \begin{proof}
   For simplicity, write $A=\{r>0: u\in r W_{\alpha,\beta}\}$, $\alpha_u=\inf A$, $B=\{\Sigma_{i=1}^{n}p(x_i)q(y_i), x_i,y_i\geq 0, |u|\leq \Sigma_{i=1}^{n}x_i\otimes y_i\}$ and $\beta_u=\inf B$. We need to show that $\alpha_u=\beta_u$. 
   
   Take $r\in A$. Then, $\frac{u}{r}\in W_{\alpha,\beta}$. So, we can write $\frac{u}{r}=\Sigma_{i=1}^{n}\lambda_i u_i$ in which $\Sigma_{i=1}^{n}|\lambda_i|\leq 1$ and $u_i\in Sol(U_{\alpha}\otimes V_{\beta})$. So, we can find $x_i\in U_{\alpha}$ and $y_i\in V_{\beta}$ with $|u_i|\leq |x_i \otimes y_i|=|x_i|\otimes |y_i|$. Therefore, we have
   \[\frac{|u|}{r}=|\Sigma_{i=1}^{n}\lambda_i u_i|\leq \Sigma_{i=1}^{n}|\lambda_i||u_i|\leq \Sigma_{i=1}^{n}|\lambda_i x_i|\otimes |y_i|.\]
   On the other hand, $\Sigma_{i=1}^{n}p_{\alpha}(\lambda_i x_i)q_{\beta}(y_i)\leq \Sigma_{i=1}^{n}|\lambda_i|p_{\alpha}(x_i)q_{\beta}(y_i)\leq 1$ so that $r\in B$. This means that $A\subseteq B$ and so $\beta_u\leq \alpha_u$. For the converse, assume that there are $x_i\geq 0$ in $E$ and $y_i\geq 0$ in $F$ with $|u|\leq \Sigma_{i=1}^{n}x_i\otimes y_i$. Put $r=\Sigma_{i=1}^{n}p_{\alpha}(x_i)q_{\beta}(y_i)$. We show that $r\in A$. To this end, since $W_{\alpha,\beta}$ is solid, we just need to prove that $\Sigma_{i=1}^{n}\frac{x_i\otimes y_i}{r}\in W_{\alpha,\beta}$. We have
    
   \[\Sigma_{i=1}^{n}\frac{x_i\otimes y_i}{r}=\Sigma_{i=1}^{n}\frac{p_{\alpha}(x_i)q_{\beta}(y_i)}{r}\frac{x_i}{p_{\alpha}(x_i)}\otimes \frac{y_i}{q_{\beta}(y_i)}.\]
   But $\frac{x_i}{p_{\alpha}(x_i)}\in U_{\alpha}$ and $\frac{y_i}{q_{\beta}(y_i)}\in V_{\beta}$ for each $i$. Furthermore, $\Sigma_{i=1}^{n}\frac{p_{\alpha}(x_i)q_{\beta}(y_i)}{r}\leq 1$. This means that $r\in A$ so that $\alpha_u\leq \beta_u$. This completes the proof.
   \end{proof}
   \begin{corollary}
   Suppose $E$ and $F$ are locally convex-solid vector lattices with the generating families of Riesz seminorms $(p_{\alpha})_{\alpha\in I}$ and  $(q_{\beta})_{\beta \in J}$, respectively. Then, for each $x_0\in E$ and for each $y_0\in F$, we have
   \[(p_{\alpha}\otimes q_{\beta})(x_0\otimes y_0)=p_{\alpha}(x_0)q_{\beta}(y_0).\]
      \end{corollary}
      \begin{proof}
      By Theorem \ref{3}, it is obvious that  $(p_{\alpha}\otimes q_{\beta})(x_0\otimes y_0)\leq p_{\alpha}(x_0)q_{\beta}(y_0)$. For the other side, by the Hahn-Banach theorem, there exist $f\in E_{+}^{'}$ such that $f(x_0)=p_{\alpha}(x_0)$ and $|f(x)|\leq p_{\alpha}(x)$  for each $x\in E$. Moreover, there is $g\in F{+}^{'}$ with $g(y_0)=q_{\beta}(y_0)$ and $|g(y)|\leq q_{\beta}(y)$ for each $y\in F$. By considering Theorem \ref{3}, assume that $|x_0\otimes y_0|=|x_0|\otimes |y_0|\leq \Sigma_{i=1}^{N}x_i\otimes y_i$ in which $x_i,y_i\geq 0$ and $N\in \Bbb N$. Consider the positive bilinear form $B$ on $E\times F$ defined via $B(x,y)=f(x)g(y)$. So, we have
      \[B(x_0,y_0)=p_{\alpha}(x_0)q_{\beta}(y_0)\leq B(|x_0|,|y_0|)\leq \Sigma_{i=1}^{N}B(x_i,y_i)=\Sigma_{i=1}^{N} f(x_i)q(y_i)\leq\Sigma_{i=1}^{N}p_{\alpha}(x_i)q_{\beta}(y_i).\]
      Since this happens for every representation, we conclude that $p_{\alpha}(x_0)q_{\beta}(y_0)\leq (p_{\alpha}\otimes q_{\beta})(x_0\otimes y_0)$. This completes the proof.
    
        \end{proof}
   \begin{corollary}
   Suppose $(E,\tau_1)$ and $(F,\tau_2)$ are Hausdorff locally convex-solid vector lattices. Then, $(E\overline{\otimes}F,\tau_F)$ is a Hausdorff locally convex-solid vector lattice.
   \end{corollary}
   \begin{proof}
   Assume that $\tau_1$ is generated by a family $(p_{\alpha})_{\alpha\in I}$ of Riesz seminorms and $\tau_2$ possesses a family $(q_{\beta})_{\beta \in J}$ of  generating Riesz seminorms. By the hypotheses, both families are separating. Suppose $0\neq u\in E\overline{\otimes}F$. By \cite[1(A) e]{Fremlin:74}, there exists $x_0>0$ in $E$ and $y_0>0$ in $F$ with $|u|\geq x_0\otimes y_0$. Thus,
   \[(p_{\alpha}\otimes q_{\beta})(u)=(p_{\alpha}\otimes q_{\beta})(|u|)\geq (p_{\alpha}\otimes q_{\beta})(x_0\otimes y_0)= p_{\alpha}(x_0)q_{\beta}(y_0)\neq 0,\]
   as claimed.
   \end{proof}
   Finally, we establish a universal property for the Fremlin tensor product of locally convex-solid vector lattices with a topological flavor. Just note that if $X$ and $Y$ are vector spaces, $T:X\to Y$ is a linear operator and $A\subseteq X$, then, it can be verified easily that $T(Conv_b(A))=conv_b(T(A))$. 
   \begin{proposition}
   Suppose $(E,\tau_1)$ and $(F,\tau_2)$ are locally convex-solid vector lattices. Then, for each locally convex-solid vector lattice $G$ and each continuous lattice bimorphism $\Phi:E\times F\to G$, there exists a unique continuous lattice homomorphism $T$ from $(E\overline{\otimes}F,\tau_F)$ into $G$ such that $\Phi=T \otimes$. 
   \end{proposition}
   \begin{proof}
   First, suppose $(U_{\alpha})_{\alpha\in I}$ and $(V_{\beta})_{\beta \in J}$ are bases for $\tau_1$ and $\tau_2$, respectively consisting of convex-solid zero neighborhoods. By \cite[Theorem 4.2 (ii)]{Fremlin:72}, there exists a unique lattice homomorphism $T:E\overline{\otimes}F\to G$ such that $\Phi=T\otimes$. So, it is enough to show that $T$ is continuous. Assume that $W$ is an arbitrary convex-solid zero neighborhood in $G$. By the continuity of $\Phi$, there are some $\alpha_0$ and $\beta_0$ with $\Phi(U_{\alpha_0},V_{\beta_0})\subseteq W$. Now, by using Lemma \ref{1}, we have 
   \[T(W_{\alpha_0,\beta_0})=T(Conv_b(Sol(U_{\alpha_0}\otimes V_{\beta_0})))\subseteq Conv_b(Sol(T(U_{\alpha_0}\otimes V_{\beta_0})))=\]\[Conv_b(Sol(\Phi(U_{\alpha_0},V_{\beta_0})))\subseteq Conv_b(Sol(W))=W.\]
   This proves the claim. 
   \end{proof}
   
\end{document}